\documentclass[a4paper,10pt,reqno]{amsart}
\usepackage{dsfont,amsmath,amsfonts,amscd,amssymb,graphicx,mathrsfs}

\usepackage[usenames]{color}
\usepackage{setspace}
\usepackage{array,multirow,booktabs,longtable}
\RequirePackage{tikz-cd}
\usepackage{tikz}
\usepackage{array}       

\DeclareMathOperator{\codim}{codim} 
 
\definecolor{lblue}{rgb}{0.3,0.0,4.4}
\definecolor{lred}{rgb}{4.3,0.0,0.4}

\newcommand{\Q}{\mathbb{Q}} 
 
\newcommand{\sm}{\mathrm{sm}}

\renewcommand{\o}{\otimes}

\newcommand{\Z}{\mathbb{Z}}

\newcommand{\pri}{\mathrm{pri}}

\newcommand{\C}{\mathbb{C}}

\newcommand{\mc}{\mathcal}

\newcommand{\ms}{\mathscr}

\newcommand{\ds}{\dots}

\newcommand{\Cs}{\mathbb{C}^{\times}}

\newcommand{\mf}{\mathfrak}

\newcommand{\git}{\ensuremath{/\!\!/\!}}

\newtheorem{thm}{Theorem}[section]
\newtheorem{lem}[thm]{Lemma}
\newtheorem{prop}[thm]{Proposition}
\newtheorem{cor}[thm]{Corollary}

\theoremstyle{definition}
\newtheorem{defn}[thm]{Definition}
\newtheorem{example}[thm]{Example}

\newtheorem{remark}[thm]{Remark}
\newtheorem{question}[thm]{Question}

 \definecolor{lightblue}{rgb}{0.8,0.8,0.9}
  \definecolor{lightred}{rgb}{0.9,0.8,0.8}
\begin{document}

\title{On symplectic resolutions and factoriality of Hamiltonian reductions} 

\author{Gwyn Bellamy}

\address{School of Mathematics and Statistics, University Place, University of Glasgow, Glasgow,  G12 8QQ, UK.}
\email{gwyn.bellamy@glasgow.ac.uk}

\author{Travis Schedler}

\address{Department of Mathematics, Imperial College London, South Kensington Campus, London, SW7 2AZ, UK}
\email{trasched@gmail.com}

\begin{abstract}
  Recently, Herbig--Schwarz--Seaton have shown that $3$-large
  representations of a reductive group $G$ give rise to a large class
  of symplectic singularities via Hamiltonian reduction. We show that
  these singularities are always terminal. We show that they are
  $\Q$-factorial if and only if $G$ has finite abelianization. When
  $G$ is connected and semi-simple, we show they are actually locally
  factorial.  As a consequence, the symplectic singularities do not
  admit symplectic resolutions when $G$ is semi-simple. We end with some
  open questions.
\end{abstract}

\maketitle

 
\section{Introduction}\label{sect:intro}

Hamiltonian reduction is an extremely powerful technique, in both physics and differential geometry, for producing rich new symplectic manifolds from a manifold with Hamiltonian $G$-action. The same technique also works well in the algebraic setting, except that the resulting spaces are often singular, and hence cannot be (algebraic) symplectic manifolds. 

Thanks to Beauville \cite{Beauville}, there is an effective generalization of algebraic symplectic manifold to the singular setting, appropriately called ``symplectic singularities''. Often, these singularities admit symplectic resolutions, i.e., Poisson resolutions of singularities by symplectic varieties.  Such resolutions have  become very interesting from multiple points of view: representation theory (of quantizations), 3-D physical mirror symmetry, algebraic and symplectic geometry, and so on. Note that, in order to admit a symplectic resolution, a variety must be a symplectic singularity, but the converse is not true.

Thus, it is natural to ask if algebraic Hamiltonian reduction gives rise to spaces with symplectic singularities. In general, examples make it clear that the answer is sometimes yes, sometimes no. For instance, examples show that even beginning with a symplectic linear representation of $G$, the resulting Hamiltonian reduction can be non-reduced or reducible; even if reduced and irreducible, it is often not normal. On the other hand many interesting classes of examples, such as Nakajima's quiver varieties,   give rise to symplectic singularities \cite{BellSchedQuiver}.   

Recently, Herbig--Schwarz--Seaton \cite{SchwarzHamred} have shown that linear $G$-representations $V$ satisfying a mild technical condition  (the $3$-large representations) 
give rise to Hamiltonian reductions (of $T^*V$ by $G$) which 
 have symplectic singularities. 

This leads to the natural question:
\begin{center}
Do these symplectic singularities admit symplectic resolutions?
\end{center}
We prove two key results (Theorem \ref{thm:mainshort} and Corollary \ref{cor:nosympres}) in this direction. \\

First, we introduce some notation. Let $G$ be a reductive (possibly disconnected) algebraic group over $\C$ and $V$ a finite dimensional $G$-representation. For each integer $k \ge 0$, one has the notion of a $k$-large representation, which roughly measures the codimension of points where certain undesirable behaviors occur (the orbit is not closed, the stabilizer is not minimal, or the stabilizer has a given positive dimension). We recall it precisely in Definition \ref{d:2large} below.
  In particular, as explained in \cite{SchwarzHamred}, if $G$ is connected and simple then all but finitely many $G$-representations $V$, with $V^G = \{ 0 \}$, are $3$-large; for a more general statement with $G$ connected and semi-simple see \cite[Theorem 3.6]{SchwarzHamred}. 

The representation $W := V \times V^*$ has a canonical $G$-invariant symplectic $2$-form $\omega$ such that the action of $G$ on $W$ is Hamiltonian, with moment map $\mu : W \rightarrow \mf{g}^*$ given by 
$$
\mu(v,\lambda) (x) = \lambda(x \cdot v), \quad \forall \ (v,\lambda) \in V \times V^*,  x \in \mf{g}.  
$$
The associated (algebraic) Hamiltonian reduction is the GIT quotient $\mu^{-1}(0) \git \, G$. We recall from \cite{Beauville} that a variety $X$ is said to be a symplectic singularity if it is normal, its smooth locus has a symplectic $2$-form $\omega$, and for any resolution of singularities $\rho : Y \rightarrow X$, the rational $2$-form $\rho^* \omega$ is regular. Moreover, $\rho$ is said to be a symplectic resolution if the $2$-form $\rho^* \omega$ is also non-degenerate. In particular, this makes $Y$ an algebraic symplectic manifold. 

\begin{thm}[\cite{SchwarzHamred}]\label{thm:hss}
	If $V$ is $3$-large then $\mu^{-1}(0) \git \, G$ is a symplectic singularity. 
\end{thm}
	    
\begin{proof}
	The definition of Hamiltonian reduction used in \cite{SchwarzHamred} is different from the one given above. However, it follows from \cite[Lemma 2.8]{SchwarzHamred} that the two definitions coincide if $V$ is $2$-large. Therefore the result follows from \cite[Theorem 1.1]{SchwarzHamred}.
\end{proof}

Recall that the \emph{abelianization} of $G$ is $G_{\text{ab}} := G/[G,G]$. The group $G$ is called \emph{perfect} if $G_{\text{ab}}=\{1\}$, i.e., $G=[G,G]$. We will show (Corollary \ref{cor:QfactorialHamred} and Proposition \ref{prop:terminal}):

\begin{thm}\label{thm:mainshort}
	Let $G$ be a reductive group acting on a $3$-large representation $V$ and $X := \mu^{-1}(0) \git \, G$ the associated Hamiltonian reduction. 
	\begin{enumerate}
		\item[(a)] $X$ has terminal singularities. 
		\item[(b)] $X$ is $\Q$-factorial if and only if $G_{\text{ab}}$ is finite.
		\item[(c)] If $G$ is perfect then $X$ is locally factorial.
	\end{enumerate} 
\end{thm}

In particular, if $G$ is connected and semi-simple then $X$ is locally factorial. Remarkably, the theorem provides an example of a symplectic singularity that is not $\Q$-factorial, but whose quotient by $\Z_2$ is $\Q$-factorial; see example \ref{ex:strangleexample}. 


\begin{cor}\label{cor:nosympres}
	Let $V$ be a $3$-large representation of $G$. If $G_{\text{ab}}$ is finite and $G$ acts non-trivially on $V$, then the symplectic singularity $X$ does not admit a symplectic resolution. 
\end{cor}
Note that, if the connected component $G^\circ$ of the identity is semi-simple,  $G_{\text{ab}}$ is a quotient of $\pi_0(G)$, which is finite. Also, when $\dim G > 0$, the assumption that $G$ acts non-trivially on $V$ is unnecessary, as it follows from the $3$-large property. 
\begin{proof}[Proof of Corollary \ref{cor:nosympres}]
	As we show in Lemma \ref{lem:2largeXsing} below, the fact that $V$ is $3$-large and $G$ acts non-trivially on $V$ forces $X$ to be singular. The fact that $X$ is $\Q$-factorial by Theorem \ref{thm:mainshort}, together with van der Waerden purity, implies that if $\rho : Y \rightarrow X$ is a symplectic resolution then the exceptional locus on $Y$ is a divisor. But since $X$ has terminal singularities, any
crepant resolution must have exceptional locus of codimension at least two. This contradicts the fact that every symplectic resolution is crepant.
\end{proof}	

Theorem \ref{thm:mainshort}, combined with Namikawa's result \cite[Theorem 5.5]{Namikawa}, implies that:

\begin{cor}
	If $G_{\text{ab}}$ is finite, then all Poisson deformations of the reduction $X$ are locally trivial as ordinary deformations. In particular, the singularities cannot change under Poisson deformation. 
\end{cor}
In Section \ref{s:disconn}, we explain how the above results
generalize, for finite groups, to the case when $V$ is not linear. One
simple consequence is that, by considering the finite quotient
$(\mu^{-1} \git \, G^\circ) / (G/G^\circ)$, one can reduce Theorem
\ref{thm:mainshort}.(b) to the connected case (although we do not need
this).

 Finally, in Section \ref{s:openq}, we present some open
 questions. For example, for $\dim G > 0$, can one generalize the
 results above to the case where $V$ is not linear? What happens if
 one replaces the affine quotient by a GIT quotient, or when one takes
 Hamiltonian reduction at a nonzero character of $\mathfrak{g}$?

\section{Hamiltonian reductions}

We assume throughout this section that $G$ is a reductive (possibly disconnected) algebraic group over $\C$. Let $N$ be an irreducible affine $G$-variety. Let $k = \min \{ \dim G_x : G \cdot x \textrm{ is closed} \}$ and let $l$ be the minimum number of connected components of $G_x$ as $x$ ranges over all points of $N$ with $G \cdot x$ closed and $\dim G_x = k$. Write $N'$ for the set of all points in $N$ such that the number of connected components of $G_x$ is $l$, the orbit $G \cdot x$ is closed, and $\dim G_x = k$. The orbits in $N'$ are the \textit{principal orbits}. Following G.~Schwarz, when $k=0$ we say that $V$ has finite principal isotropy groups (FPIG). If the categorical quotient of $N$ is $\xi: N \rightarrow X := N \git \, G$, then $X_{\mathrm{pri}} := \xi(N')$ and $N_{\mathrm{pri}} := \xi^{-1}(X_{\mathrm{pri}})$. Since $X$ is irreducible,  $X_{\mathrm{pri}}$ is open and dense in $X$, and it is a consequence of Luna's slice theorem that the isotropy groups of all points in $N'$ are conjugate; see \cite[\S 1.4]{Schwarz}.  These groups are called \emph{principal isotropy groups}.

Let us recall the definition of $k$-large representations (restated slightly from  \cite[\S 2.1]{SchwarzHamred}):
\begin{defn} \label{d:2large}
	A representation $V$ of $G$ is $k$-large if: 
	\begin{itemize}
		\item $V$ has FPIG;
		\item $\codim V \setminus V' \geq k$ (``$k$-principal'')
		\item $\codim V_{(r)}\geq r+k$ for $1 \leq r \leq \dim G$, where $V_{(r)} := \{v \in V \mid \dim G_v = r\}$ (``$k$-modular'').
	\end{itemize}
\end{defn}

Observe that if $V$ has FPIG, then $V_{\mathrm{pri}}=V'$ consists precisely of the principal orbits.   We will need the following result.
 \begin{lem}\label{lem:VpriVstarVVstarpri}
  If $V$ is $k$-large for $k \geq 2$  then $V \times V^*$ is $2k$-large.
Moreover, $V' \times V^* \subseteq (V \times V^*)'$; similarly
 $V \times (V^*)' \subseteq (V \times V^*)'$.
 \end{lem}

\begin{proof}
  Note that, for $v \in V$ and $f \in V^*$, we have
  $\dim G_{(v,f)} \leq \text{min} \{\dim G_v, \dim G_f\}$. Thus the
  $k$-modularity of $V$ implies $2k$-modularity of $V \times V^*$ (in
  fact, $(2k+1)$-modularity).  To prove the $2k$-principal property and the FPIG condition, it suffices to prove the final assertion,
  which we do in the remainder of the proof. (This also shows that we can replace the primes by subscripts ``pri''.)

 Since $V$ is $k$-large with $k \ge 2$, it follows from \cite[Corollary 7.7]{Schwarz}
 that the principal isotropy groups $G_v, v \in V_{\mathrm{pri}}$ are all
 equal to the kernel $K$ of the action of $G$ on $V$.
 Since $K$ is also the kernel of the action on $V \times V^*$, we have
 $K < G_{(v,f)}$ for all $v \in V, f \in V^*$.  On the other hand, if
 $v \in V'$, then $K > G_{(v,f)}$ for all
 $f \in V^*$.  So $K=G_{(v,f)}$ for all $v \in V', f \in V$.

We claim that all orbits in $V' \times V^*$ are closed. More
generally, let $W$ be any representation of $G$ and $w \in W$. Then
$G \cdot (v, w)$ has dimension $G$, as $G \cdot v$ does.  If
$G \cdot (v,w)$ is not closed, then its boundary contains an orbit of
the form $G \cdot (v,w')$, as $G \cdot v$ itself is closed. Being on
the boundary, the orbit has dimension strictly less than $\dim
G$. This contradicts the previous statement. The claim follows.

Thus the kernel $K$ is also the principal isotropy group for all
points $(v,f) \in (V \times V^*)'$, which includes $V' \times V^*$,
and similarly also $V \times (V^*)'$.  This proves the final
assertion, and hence the lemma. \qedhere

\end{proof} 	

Recall that if $D_1$ and $D_2$ are Weil divisors on a normal variety $X$ then $\mc{O}(D_i)$ denotes the corresponding reflexive rank one subsheaf of $\ms{K}(X)$ and $ \mc{O}(D_1 + D_2) = (\mc{O}(D_1) \o \mc{O}(D_2))^{\vee \vee}$. 
Set $\mc{O}(D)^{(n)} := \mc{O}(nD)$.

\begin{lem}\label{lem:Nagataequivstatement}
	Let $X$ be a normal irreducible variety and $x \in X$. The following are equivalent:
	\begin{enumerate}
		\item[(i)] The local ring $\mc{O}_{X,x}$ has torsion class group $\mathrm{Cl}(\mc{O}_{X,x})$.
		\item[(ii)] For every line bundle $M_0$ on $X_{\sm}$, there exists an open subset $U$ containing both $x$ and $X_{\sm}$, and $n \geq 1$, such that $M_0^{\otimes n}$ extends to a
		line bundle $M$ on $U$.
	\end{enumerate}
\end{lem}

\begin{proof}
		Recall that $\mc{O}_{X,x}$ is a unique factorization domain if and
	only if every height one prime is principal. Geometrically, this means
	that for every hypersurface $C$ of $X$, the sheaf of ideals $\mc{I}_C$
	is free at $x$. Since $X$ is normal, $X_{\sm}$ has complement of codimension at least $2$ in $X$. 
	
	(i) implies (ii). We denote by the same symbol $M_0$ its push-forward to $X$. Thus, $M_0$ is a reflexive rank one sheaf. There exists some $n \geq 1$ such that $M_0^{(n)}$ has trivial image in $\mathrm{Cl}(\mc{O}_{X,x})$. Thus, $M := M_0^{(n)}$ is locally free in a neighborhood of $x$ and $M |_{X_{\sm}} = M_0^{\otimes n}$. 
	
	(ii) implies (i). Let $E \in \mathrm{Cl}(\mc{O}_{X,x})$. By Nagata's Theorem, we can choose a Weil divisor $D$ on $X$ whose image in $\mathrm{Cl}(\mc{O}_{X,x})$ equals $E$. Let $\mc{O}(D)$ be the corresponding reflexive rank one sheaf. We wish to show that $\mc{O}(D)^{(n)}$ is free in a neighborhood of $x$ for some $n \geq 1$. Let $M$ be the extension of $\mc{O}(D)|_{X_{\sm}}^{\o n}$ to $U$. The line bundle $M$ corresponds to a Cartier divisor $C$ on $U$; $M = \mc{O}_{U}(C)$. Then,
	$$
	(\mc{O}(D)|_{X_{\sm}})^{\o n} = \mc{O}_{X_{\sm}}(C \cap X_{\sm}), 
	$$ 
	and the divisors $nD |_{X_{\sm}}$ and $C |_{X_{\sm}}$ are linearly equivalent. Since $X$ is normal, we have $nD \sim C$, implying that $nD$ is Cartier. Thus, $nE = 0$. 
\end{proof}

The following is a variant of \cite[Theorem 6.7]{BellSchedQuiver}, itself based on a result of Drezet \cite{Drezet}. 

\begin{thm}\label{thm:Drezettorsion}
	Let $N$ be an affine locally factorial normal irreducible $G$-variety with good quotient $\xi: N \rightarrow X := N \git \, G$. Assume that:
	\begin{enumerate}
		\item[(a)] $N$ has FPIG,
		\item[(b)] the complement to $N_{\mathrm{pri}}$ in $N$ has codimension at least two; and 
		\item[(c)] the complement to $\xi^{-1}(X_{\sm})$ in $N$ has codimension at least two. 
	\end{enumerate}
	Let $x \in X$ and $y \in \xi^{-1}(x)$ such that $G \cdot y$ is closed in $N$. The following are equivalent:
	\begin{enumerate}
		\item[(i)] The local ring $\mc{O}_{X,x}$ has torsion class group $\mathrm{Cl}(\mc{O}_{X,x})$.
		\item[(ii)] For every line bundle $M_0$ on $X_{\sm}$, there exists an open subset $U$ containing both $x$ and $X_{\sm}$, and $n \geq 1$, such that $M_0^{\otimes n}$ extends to a
		line bundle $M$ on $U$.
		\item[(iii)] For every $G$-equivariant line bundle $L$ on $N$, the action of the stabilizer $G_y$ on every fiber $L_y$ factors through a finite group.
	\end{enumerate}
\end{thm} 

\begin{proof}
	Since $N$ is normal, so too is $X$. Therefore the fact that (i) is equivalent to (ii) follows from Lemma \ref{lem:Nagataequivstatement}. 
	
	The set $N_{\mathrm{pri}}$ is the pre-image under $\xi$ of $X_{\mathrm{pri}}$.  The fibers of $\xi: N \to X$ have dimension $G$ over the principal locus, hence have dimension $\geq \dim G$ everywhere. 
Therefore, the fact that the complement to $N_{\mathrm{pri}}$ in $N$ has codimension at least two implies that the complement to $X_{\mathrm{pri}}$ in $X$ has codimension at least two as well. Let $X_s = X_{\mathrm{pri}} \cap X_{\sm}$, an open set with complement of codimension at least two. Let $N_s = \xi^{-1}(X_s)$. Our assumptions imply that the complement to $N_s$ in $N$ has codimension at least two as well.   
	
	(ii) implies (iii). Suppose that $L$ is a $G$-equivariant line bundle on $N$. Since $N$ has FPIG, all stabilizers $G_y$ for $y \in N_{\mathrm{pri}}$ are conjugate. In particular, their orders are the same. Thus, there exists some $m$ for which the stabilizers all act trivially on the fibers of $L^{\o m} |_{N_s}$. By descent \cite[Theorem 1.1]{Drezet}, the line bundle $(L^{\o m}) |_{N_{s}}$ descends to a line bundle $M_0$ on $X_{s}$. This line bundle extends to $X_{\sm}$ since $X_{\sm} \setminus X_s$ has codimension at least two, and $X_{sm}$ is smooth (hence locally factorial). By (ii), there is an extension $M$ of $M_0^{\o n}$ to $U$. Then the $G$-equivariant line bundle $\xi^* M$ agrees with $L^{\o nm}$ on $N_s$. By normality, this implies that $\xi^* M = L^{\o mn}$ on $\xi^{-1}(U)$. In	particular, since $y \in \xi^{-1}(U)$, the stabilizer of $y$ acts trivially on $L_y^{\o mn}$.
	
	(iii) implies (ii). Let $M_0$ be a line bundle on $X_s$. By \cite[Lemma 6.6]{BellSchedQuiver}, the line bundle $\xi^* M_0$ extends to a $G$-equivariant line bundle $L$ on $N$. Let $y \in N$. Then, $G_y$ acts trivially on $L^{\o n}_y$ for some $n \geq 1$ (we can take $n$ to be the size of the finite quotient through which $G_y$ acts). By \cite[Lemma 6.8]{BellSchedQuiver} there is an affine open neighborhood $U$ of $x$ such that $G_{y'}$ acts trivially on $L_{y'}^{\o n}$ for all $y' \in \xi^{-1}(U)$ such that $G \cdot y'$ is closed in $N$. We may assume without loss of generality that $X_s \subset U$. Then, by descent \cite[Theorem 1.1]{Drezet}, there exists a line bundle $M$ on $U$ such that $\xi^* M \simeq L^{\o n}$. In particular, $M$ extends $M_0^{\o n}$.
\end{proof}

\begin{cor} \label{cor:drezettorsion}
Assume that $\mathrm{(a)}$--$\mathrm{(c)}$ of Theorem \ref{thm:Drezettorsion} hold, and that $N$ admits a $\Cs$-action, commuting with the action of $G$, contracting all points to a unique fixed point. If $n:=|G_{\mathrm{ab}}|$ is finite then for each Weil divisor $D$ on $N \git \, G$, $n D$ is Cartier. 
\end{cor}

\begin{proof} 
	Let $o$ be the unique fixed point of the $\Cs$-action on $N$. Then $G_o = G$ and $\{ o \}$ is a closed orbit in $N$. Let $L$ be a $G$-equivariant line bundle on $N$, as in the proof of (iii)$\Rightarrow$(ii) in Theorem \ref{thm:Drezettorsion}. Our assumptions imply that $G=G_o$ acts trivially on the fiber $L_o^{\otimes n}$. It follows that the class group of the local ring $\mc{O}_{X,\xi(o)}$ is $n$-torsion. By \cite{OpenFactorial}, this implies that for each Weil divisor $D$ in a neighborhood of $\xi(o)$ in $X$, $n D$ is Cartier. Using the contracting $\C^*$ action on $X$, this must hold globally. 
\end{proof}
\begin{remark} \label{rem:drezet-qfact}
In Theorem \ref{thm:Drezettorsion} and Corollary \ref{cor:drezettorsion}, it actually suffices to allow $N$ to be $\Q$-factorial: it need not be locally factorial. In Corollary \ref{cor:drezettorsion}, the revised statement should be that, if $m \text{Weil}(N) \subseteq \text{Cartier}(N)$, then $m|G_{\text{ab}}| \text{Weil}(N \git \, G) \subseteq \text{Cartier}(N \git \, G)$. This only affects the argument of (iii)$\Rightarrow$(ii) of Theorem \ref{thm:Drezettorsion} by replacing $\xi^* M_0$ there by $(\xi^* M_0)^{\otimes m}$.
\end{remark}

In particular, if $G$ is perfect in Corollary \ref{cor:drezettorsion}, then $N \git \, G$ is locally factorial. This applies for instance when $G$ is connected semi-simple. 

We wish to apply the above results to the particular case where $N = \mu^{-1}(0) \subset T^* V$ for some $G$-representation $V$ and $X = N \git G$. We require a technical lemma:

\begin{lem}\label{lem:pricialcodim4}
	If $V$ is $2$-large then $N \smallsetminus N_{\mathrm{pri}}$ has codimension at least $2$ in $N$. 
\end{lem}

\begin{proof}
	
	Since $V \times \{ 0 \} \subset N$, Lemma \ref{lem:VpriVstarVVstarpri} implies that $N$ contains principal points of $V \times V^*$. Thus, $N$ satisfies FPIG and 
$N_{\mathrm{pri}} = (V \times V^*)_{\mathrm{pri}} \cap N$. Moreover, 
Lemma \ref{lem:VpriVstarVVstarpri} implies that it suffices to show
that the complement to $(V_{\mathrm{pri}} \times V^*) \cap N$ in $N$
has codimension at least two. Explicitly, for each irreducible 
component $Z \subseteq V \setminus V_{\pri}$, 
we need to find a pair of
functions $f_1,f_2 \in \C[V]$, both vanishing on $Z$,
 which form
a regular sequence on $N$.

To find the functions $f_1, f_2$, note that  
$V_{\pri}$ is the preimage of an open dense subset of
$V \git \, G$, with complement of codimension at least two. Therefore there exist
$G$-invariant $f_1, f_2 \in \C[V \git \, G] = \C[V]^G$, vanishing on $Z \subseteq (V \setminus V_{\pri})$,
which are not scalar multiples of each other. Since $Z$ is irreducible,
we can assume that $f_1$ is an irreducible element of $\C[V]^G$. After replacing $f_2$ by $f_2/\gcd(f_1,f_2)$, we can also assume they share no common factors, i.e., they form a regular sequence on $V$. Then,  
it follows from \cite[Lemma 9.7]{Schwarz} that
$f_1, f_2, f_{A_1}, \ldots, f_{A_\ell}$ form a regular sequence, where
$f_{A_1}, \ldots, f_{A_\ell}$ are the defining equations for
$N$. Thus $f_1$ and $f_2$ also define a regular sequence on $N$. \qedhere

\end{proof}

For the remainder of this section, we assume that $V$ is a $3$-large representation of $G$. Let $N := \mu^{-1}(0)$ and $X := N \git \, G$. By \cite[Proposition 3.2]{SchwarzHamred}, this implies that $N$ is reduced, irreducible, and normal.
Since $V$ has FPIG by assumption and $V \times \{ 0 \} \subset N$, $N$ also has FPIG.

\begin{cor}\label{cor:QfactorialHamred}
	The Hamiltonian reduction $X$ is $\Q$-factorial if and only if the abelianization $G_{\mathrm{ab}}$ of $G$ is finite. If $G$ is perfect then $X$ is locally factorial.
\end{cor}   

\begin{proof}
	As noted in \cite[Section 3.1]{SchwarzHamred}, if $V$ is $n$-large, for $n \ge 2$, then it follows from \cite[Proposition 6]{Avr-cisa} and \cite[Remark 2.4]{SchwarzKoszul}  that $\C[N]$ is a unique factorization domain. In particular, $N$ is locally factorial.
	
	The fact that $V$ is $3$-large implies by \cite[Theorem 3.21]{SchwarzHamred} that $X_{\mathrm{sm}} = X_{\mathrm{pri}}$. Thus, $N_{\mathrm{pri}} = \xi^{-1}(X_{\sm})$. Hence, Lemma \ref{lem:pricialcodim4} implies that assumptions (a)--(c) of Theorem \ref{thm:Drezettorsion} hold in this case. 
	Note that $X$ carries a contracting $\Cs$-action, with unique fixed point $o$. Therefore, by Corollary \ref{cor:drezettorsion}, if $G$ has finite abelianization, then $X$ is $\Q$-factorial, and if it is perfect, then $X$ is locally factorial.
	
	
	Assume now that $G_{\text{ab}}$ is not finite. Then we can choose a surjective character $\theta: G \rightarrow \Cs$. In particular, $\theta^n \neq 1$ for all $n \geq 1$. Let $L$ be the $G$-equivariant line bundle on $N$ corresponding to the $(\C[N],G)$-module $\C[N] \o \theta$, where $G$ acts diagonally. Forgetting the equivariant structure, $L$ is the trivial line bundle. However, $G$ acts on the fiber $L_0$ as multiplication by $\theta$. In particular, this action does not factor through any finite group. Thus, we deduce from Theorem \ref{thm:Drezettorsion} that $\mathrm{Cl}(\mc{O}_{X,o})$ is not torsion.    
\end{proof}


The following proposition completes the proof of Theorem \ref{thm:mainshort}. 

\begin{prop}\label{prop:terminal}
	The variety $X$ has terminal singularities. 
\end{prop}

\begin{proof}
	Since we have assumed that $V$ is $3$-large, \cite[Theorem 3.21]{SchwarzHamred} says that $X_{\mathrm{sm}} = X_{\mathrm{pri}}$. Then it is a consequence of Theorem 4.4 of \textit{loc. cit.} says that the subvariety $X \smallsetminus X_{\mathrm{sm}}$ has codimension at least four in $X$. Moreover, Corollary 4.5 of \textit{loc. cit.} says that $X$ is a symplectic singularity. Therefore, it follows from \cite{NamikawaNote} that $X$ has terminal singularities.
\end{proof}

Finally, we note that:

\begin{lem}\label{lem:2largeXsing}
	If $G$ acts non-trivially on $V$ then the variety $X$ is singular. 
\end{lem}

\begin{proof}
	The assumption that $G$ acts non-trivially on $V$ implies that
	$X_{\mathrm{pri}} \neq X$ because $0 \notin X_{\mathrm{pri}}$. Then the claim once again follows from \cite[Theorem 3.21]{SchwarzHamred}, which says that $X_{\mathrm{pri}} = X_{\sm}$.  
\end{proof}

The reader can check that Theorem \ref{thm:hss}, Theorem \ref{thm:mainshort} and Corollary \ref{cor:nosympres} all hold provided $V$ is $2$-large and $X_{\mathrm{sm}} = X_{\mathrm{pri}}$. The $3$-large condition is only required to guarantee, by \cite[Theorem 3.21]{SchwarzHamred}, that $X_{\mathrm{sm}} = X_{\mathrm{pri}}$. There exist examples of $2$-large representation that are not $3$-large, but for which $X_{\mathrm{sm}} = X_{\mathrm{pri}}$. In particular, the following is explained after the proof of Lemma 5.6 of \textit{loc. cit.}

\begin{lem}\label{lem:torus1largetechnical}
	If the connected component $G^{\circ}$ is a torus and $V$ is $1$-large then $X_{\mathrm{sm}} = X_{\mathrm{pri}}$. 
\end{lem}

\begin{example}\label{ex:strangleexample}
	Let $G = \Cs \rtimes \Z_2$, where $s \in \Z_2$ acts on $\Cs$ by $s(t) = t^{-1}$. Then $[G,G] = \Cs < G$ and $G/[G,G] \cong \Z_2$ is finite.
 Let $V = \C^{2n}$ for $n \geq 2$ with coordinates $x_1, \ds, x_{2n}$ such that 
	$$
	t \cdot x_i = t x_i, \quad t \cdot x_{i + n} = t^{-1} x_{i+n}, \quad s \cdot x_i = x_{i +n}, \quad s \cdot x_{i+n} = x_i \quad \textrm{for $1 \le i \le n$.}
	$$
	Then $V$ is a $n$-large representation of $G$ and Lemma \ref{lem:torus1largetechnical} implies that $X_{\mathrm{sm}} = X_{\mathrm{pri}} = X \smallsetminus\{ 0 \}$. We deduce from Theorem \ref{thm:mainshort} that $\mu^{-1}(0) \git \, G$ is terminal and 
$\Q$-factorial. Moreover, it does not admit any symplectic resolution. 
	
	If, instead, one takes $G^{\circ} = \Cs$ acting on the same representation, then this is once again $2$-large and $X_{\mathrm{sm}} = X_{\mathrm{pri}}$. However, Theorem \ref{thm:mainshort} says that $X := \mu^{-1}(0) \git \, G^{\circ}$ is no longer $\Q$-factorial. Moreover, $X$ is isomorphic to the minimal nilpotent orbit in $\mf{sl}_{2n}$ and hence has a symplectic resolution given by $T^* \mathbb{P}(V)$; \cite[Corollary 3.19]{FuNilpotentResolutions} then provides another proof that $X$ is not $\Q$-factorial in this case. This gives an example of a symplectic singularity that is not $\Q$-factorial, but whose quotient by $\Z_2$ is $\Q$-factorial.     

One can also produce examples where $X$ is not $\Q$-factorial but $X/H$ is locally factorial for $H$ finite (and both are terminal symplectic singularities). For this let $G$ be a perfect reductive group, such as $(\Cs)^4 \rtimes A_5 < \mathrm{SL}_5$, and let $V$ be any $2$-large representation of $G$ with $X_{\mathrm{sm}} = X_{\mathrm{pri}}$, e.g., $V=\text{Res}^{\mathrm{SL}_5}_G (\C^5)^2$ in this case. Then set $X := \mu^{-1}(0) \git \, G^\circ$ and $H=G/G^\circ = \pi_0(G)$; in this example, $H=A_5$.

	
\end{example}

\section{Remarks on disconnected groups}\label{s:disconn}

We have chosen to work with disconnected groups partly since, as
illustrated by example \ref{ex:strangleexample}, it leads to strange new
behavior. In fact, it is also possible to deduce Theorem
\ref{thm:mainshort}.(b) for groups $G$ whose connected component
$G^\circ$ of the identity is semi-simple, directly from the case of
$G^\circ$. More generally, if $Y$ is an irreducible symplectic
singularity, and $H$ a finite group of symplectic automorphisms of
$Y$, by \cite[Proposition 2.4]{Beauville} $Y/H$ is also a symplectic
singularity. If $Y$ is additionally terminal, by \cite{NamikawaNote},
$Y$ has singularities in codimension at least four, and $Y/H$ is
terminal if and only if it has the same property. Thus, $Y/H$ is
terminal if and only if $Y$ is terminal and the non-free locus of $H$
on $Y$ has codimension at least four. On the other hand, if $Y$ is
$\Q$-factorial, so is $Y/H$: see, e.g., \cite[Theorem 3.8.1]{Benson}
where $Y$ and $Y/H$ need only be normal, not symplectic
singularities. In our situation, the result follows from the
$\Q$-factorial version of Theorem \ref{thm:Drezettorsion} (see Remark
\ref{rem:drezet-qfact}), specializing to finite groups. 
When $Y$ has a contracting
$\Cs$-action which commutes with $H$, then $m \cdot \text{Weil}(Y) \subseteq \text{Cartier}(Y)$ implies that $|H_{\text{ab}}| \cdot m  \cdot \text{Weil}(Y/H) \subseteq \text{Cartier}(Y/H)$ (by Corollary \ref{cor:drezettorsion}; the statement also follows from \cite[Theorem 3.8.1]{Benson}).

Put together, we see that the quotient of a $\Q$-factorial terminal
singularity by a finite group of symplectomorphisms acting freely
outside codimension at least four is also a $\Q$-factorial terminal
symplectic singularity. In particular, if such a quotient is singular
(which is true unless $Y$ is smooth and $H$ acts freely), then there
is no symplectic resolution of singularities. This generalizes, and
provides a completely different proof of, the theorem of Verbitsky
\cite{Verbitsky}, which considered the case that $Y$ is a symplectic
vector space.  (Note, though, that the nonexistence of symplectic
resolutions in the general case follows by formal localization from
Verbitsky's theorem if $H$ has nontrivial isotropy groups on the
smooth locus of $Y$).

Now suppose that $V$ is a $2$-large representation of the reductive group $G$. Set $\xi : \mu^{-1}(0) \rightarrow Y := \mu^{-1}(0) \git \, G^{\circ}$ and $H = G /
G^{\circ} K$, for $K$ the kernel of the action $G$ on $V$. By \cite[Corollary 7.7]{Schwarz}, $G/K$ acts freely on the principal
locus $V_{\pri}$, hence also on $\mu^{-1}(0)_{\pri}$.  By the proof
of \cite[Theorem 4.4]{Schwarz}, the complement to the image $U := \xi(\mu^{-1}(0)_{\pri})$ has codimension at least four. Since
$\mu^{-1}(0)_{\pri}$ consists of closed orbits,
$H$ acts freely on $U$. If in addition $G^{\circ}$ is semi-simple, then by Theorem \ref{thm:mainshort}, $Y$ is locally factorial. Then we are in the situation of the previous paragraph, so that $X:=Y/H$ is a $\Q$-factorial terminal symplectic singularity. This verifies Theorem \ref{thm:mainshort}.(b), for $G^\circ$ semi-simple, assuming only the connected case.

Note that such considerations appear insufficient for deducing Theorem \ref{thm:mainshort}.(a) from the connected case, since as Example \ref{ex:strangleexample} shows, in general if $Y$ is a non-$\Q$-factorial singularity then a finite quotient $Y/H$ can nonetheless be $\Q$-factorial (even in the case of terminal symplectic singularities).


\section{Open questions}\label{s:openq}
The above suggests the following possible generalizations:
\begin{question}\label{q:vglob}
  Suppose we replace a $3$-large representation $V$ of a reductive group $G$ by a smooth irreducible affine
  variety $Y$ such that, at every $y \in Y$ such that
  $G \cdot y \subseteq Y$ is closed, the representation $T_yY$ of
  $G_y$ is $3$-large. For moment map $\mu: T^* Y \to \mathfrak{g}^*$, do the
  analogs of Theorems \ref{thm:hss} and \ref{thm:mainshort} hold? 
 \end{question}
Next, by Lemma \ref{lem:VpriVstarVVstarpri}, if $V$ is a $k$-large
representation of $G$, then $T^* V$ is $2k$-large.
\begin{question} Suppose that a reductive group $G$ acts symplectically on a
  representation $U$ of $G$ which is now assumed to be $6$-large. For moment map
  $\mu: U \to \mathfrak{g}^*$, do the analogues of Theorems
  \ref{thm:hss} and \ref{thm:mainshort} hold for the reduction
  $\mu^{-1}(0) \git \, G$?
\end{question}
Of course, we can put the two questions together:
\begin{question}\label{q:glob-sav}
  If $U$ is a symplectic irreducible affine variety, or more generally an affine symplectic
  singularity, with a Hamiltonian action of a reductive group $G$, and
  $T_u U$ is $6$-large for every $u \in U$ with
  $G \cdot u \subseteq U$ closed, then do Theorems \ref{thm:hss} and
  \ref{thm:mainshort} hold for the Hamiltonian reduction
  $\mu^{-1}(0) \git \, G$? If $U$ is (singular and) conical with cone point $o \in U$ and $\Cs$-action commuting with the action of $G$, it is enough to ask that
$T_o U$ be $6$-large.
\end{question}

If $G$ is finite, the above questions all have affirmative answers by Section \ref{s:disconn}.  If the questions have affirmative answers in general, then whenever $G_{\text{ab}}$ is finite, suitably large Hamiltonian reductions by $G$ do not admit symplectic resolutions.

Finally, we can ask about Hamiltonian reductions at nonzero coadjoint orbits.  Recall
that, if $V$ is a $2$-large representation of a reductive group $G$, then $\mu: T^*V \to \mathfrak{g}^*$ is flat
by \cite[Proposition 9.4]{Schwarz}.
\begin{question}
Suppose that $V$ is a $2$-large representation of a reductive group $G$ and $\mu: T^* V \to \mathfrak{g}^*$ the moment map. 
Is the reduction $\mu^{-1}(\chi) \git \, G$ $\Q$-factorial for generic characters $\chi: \mathfrak{g} \to \C$?
More generally, if $U$ is a $4$-large symplectic representation of $G$ and $\mu: U \to \mathfrak{g}^*$ the moment map, is $\mu^{-1}(\chi) \git \, G$ $\Q$-factorial for generic $\chi$? The same questions apply also in the global setting (following Questions \ref{q:vglob} and \ref{q:glob-sav}).
\end{question}
If the answer is affirmative and $\mu^{-1}(0) \git \, G$ has symplectic singularities, then a symplectic smoothing exists if and only if it can
be obtained by varying the moment map parameter.
Similarly, it is also interesting to replace deformations (varying
$\chi$) as above by partial resolutions, obtained by replacing the
affine quotient above by a GIT quotient corresponding to
a character $\theta: G \to \Cs$: are the resulting quotients $\Q$-factorial for generic
$\theta$? If so, then whenever symplectic resolutions exist, they can be obtained by varying $\theta$.  The $2$-large property is important here:
\begin{example}
Let $V=(\mathfrak{sl}_2)^2$, considered as a representation of $G=PGL_2$. 
For $\mu: T^* V \to \mathfrak{g}$, by \cite{KaledinLehn},
the quotient $\mu^{-1}(0) \git \, G$ identifies with the locus of square-zero matrices in $\mathfrak{sp}_4$, and in particular is
a symplectic singularity which is not terminal. (The singular locus is the codimension-two locus of rank-one matrices in $\mathfrak{sp}_4$). In particular, $V$ is not $2$-large (in fact, it is $1$-large).
 Note that $G$ is simple, and one cannot obviously construct any symplectic resolution via GIT.  However, as explained in \cite[Remark 4.6]{KaledinLehn}, following \cite{OGr-K3} in the global situation of moduli spaces of sheaves on $K3$ surfaces, blowing up the
reduced singular locus of $X = \mu^{-1}(0) \git \, G$ produces a symplectic resolution. This is is also realized by the partial Springer resolution  with source the cotangent bundle of the Lagrangian Grassmannian in $\C^4$.  
We note that generalizations
of this construction to quiver varieties are given in \cite{BellSchedQuiver}.
\end{example}

\def\cprime{$'$} \def\cprime{$'$} \def\cprime{$'$} \def\cprime{$'$}
\def\cprime{$'$} \def\cprime{$'$} \def\cprime{$'$} \def\cprime{$'$}
\def\cprime{$'$} \def\cprime{$'$} \def\cprime{$'$} \def\cprime{$'$}
\def\cprime{$'$} \def\cprime{$'$}


\end{document}